\documentclass[11pt,epsfig,amsfonts]{amsart}
\usepackage{amssymb, amsmath, amsthm}
\usepackage[colorlinks=true,linkcolor=blue,citecolor=red]{hyperref}
\usepackage{color}
\usepackage{epsfig}
\usepackage{amsmath}
\usepackage{amssymb}
\usepackage{amscd}
\usepackage{graphicx}
\numberwithin{equation}{section}
\newtheorem{lemma}{Lemma}[section]
   \newtheorem{theorem}{Theorem}
   \newtheorem{re}{Remark}[section]
   
\newcommand{\al}{\alpha}
\newcommand{\be}{\beta}

\def\spn{\mathop{\mathrm{span}}}

\usepackage{epsfig}

\numberwithin{equation}{section}
\numberwithin{theorem}{section}
\numberwithin{prop}{section}
\numberwithin{lemma}{section}
\numberwithin{re}{section}
\numberwithin{coro}{section}

\newcommand{\R}{\mathbb{R}}



\thanks{ Email: \dag liji@hust.edu.cn, \ddag yliu@uta.edu, \S {wuq@ohio.edu}. }

\begin{document}

\title[Spectral stability of Degasperis-Procesi solitary waves]{Spectral stability of smooth solitary waves for the Degasperis-Procesi Equation}
\today

\maketitle

\centerline{\scshape Ji Li$^{\,\dag}$, Yue Liu$^{\,\ddag}$ and Qiliang Wu$^{\,\S,*}$}
\medskip
{\footnotesize
\centerline{\dag. School of Mathematics and Statistics, Huazhong University of Science and Technology}
   \centerline{Wuhan, Hubei 430074, China}
 \centerline{\ddag. Department of Mathematics, University of Taxas at Arlington}
  \centerline{Arlington, TX 76019-0408, USA}
\centerline{\S. Department of Mathematics, Ohio University}
  \centerline{Athen, OH, 45701, USA}
  \centerline{$*$ the corresponding author}
   }

\begin{abstract}
The  Degasperis-Procesi equation is an approximating model of  shallow-water wave propagating mainly in one direction to the Euler equations.   Such  a model equation is analogous to the Camassa-Holm approximation of the two-dimensional incompressible and  irrotational  Euler equations with the same asymptotic accuracy,  and is integrable with the bi-Hamiltonian structure. In the present study, we establish existence and spectral stability results of localized smooth solitons to the Degasperis-Procesi equation on the real line.  The stability proof relies essentially on refined spectral analysis of the linear operator corresponding to the second-order variational derivative of the Hamiltonian of the Degasperis-Procesi equation.
\end{abstract}

\section{Introduction}
The Degasperis-Procesi (DP) equation
\begin{equation}\label{DP2k}
m_t+2ku_x+3mu_x+um_x=0,\quad x\in\R,  \;  t>0,
\end{equation}
with momentum density $m\triangleq u-u_{xx}$ and $ k > 0 $ as a parameter related to the critical shallow water speed, was originally derived by Degasperis and Procesi \cite{D-P} using the method of asymptotic integrability up to the third order as one of three equations in the
family of third-order dispersive PDE conservation laws of the form
\begin{equation*}
u_t-\alpha_2^2u_{xxt}+\alpha_2u_{xxx}+c_0u_x= \partial_x(c_1u^2+c_2u_x^2+c_3uu_{xx}).
\end{equation*}
The other two integrable equations in the family, after rescaling and applying a Galilean transformation, are the Korteweg-de Vries (KdV) equation \cite{KdV},
\begin{equation}
u_t+u_{xxx}+uu_x=0,
\end{equation}
and the Camassa-Holm(CH) shallow-water equation \cite{CH, F-F}  ({see also \cite{cola} for a rigorous justification in shallow water approximation),
\begin{equation} \label{CH}
m_t+2ku_x+2mu_x+um_x=0,\quad m=u-u_{xx}.
\end{equation}
The DP equation is also  an approximation to the incompressible Euler equations for shallow
water  and its asymptotic accuracy is
the same as that of the CH shallow-water equation \cite{cola} in the CH scaling, where the solution $ u(t, x)  $ of \eqref{DP2k} represents the horizontal velocity field at height $ z_0 = \sqrt{\frac{23}{36}}$ after the re-scaling within $ 0 \leq z_0 \leq 1$
at time t in the spatial $x-$direction with momentum density $m $.

The DP equation \eqref{DP2k} has an apparent similarity to the CH equation \eqref{CH}, and both of them are important model equations for shallow
water waves with breaking phenomena, i.e., the wave remains
bounded but its slope becomes unbounded \cite{C-Ep, C-E, L-Y, Wh}. However, there was very little   known about qualitative properties and long-time dynamics of the DP equation, and what was known about the CH equation can not be directly applied to the DP equation, due to major structural differences between the DP equation and the CH equation. For instance, the isospectral problems in the Lax pair for the DP equation \eqref{DP2k} and the  CH equation are respectively  a third-order equation \cite {D-H-H}
\[
\psi_x-\psi_{xxx}-\lambda m\psi=0,
\]
and a second-order
equation \cite {CH}
\[\psi_{xx}-\frac{1}{4}\psi-\lambda m\psi=0,
\]
where $m=u-u_{xx}$ in both cases.
Moreover, the CH equation is a re-expression of geodesic flow on
the diffeomorphism group \cite{C-K} and on the Bott-Virasoro group \cite{Mi}, while no
such geometric derivation of the DP equation is available.

When it comes to solitons, the main focus of this work, it is well-known that the KdV equation is an integrable Hamiltonian equation
that possesses smooth solitons as traveling waves. In the KdV equation, the
leading-order asymptotic balance that confines the traveling wave solitons occurs
between nonlinear steepening and linear dispersion. On the other hand, the nonlinear dispersion
and nonlocal balance in the CH equation and the DP equation, can still produce confined solitary traveling waves. There are two scenarios, though, depending on the value of $k$.

In the limiting case of vanishing linear dispersion ($k=0$), smooth solitary waves become peaked solitons, called peakons. More specifically, when $k=0$, the CH equation can be written as
\begin{equation}\label{CH0}
u_t+\partial_x(\frac{1}{2}u^2+\frac{1}{2}\phi*(\frac{1}{2}u_x^2+u^2))=0,\quad t>0, \; \, x\in\R,
\end{equation}
and the DP equation as
\begin{equation}\label{DP}
u_t+\partial_x(\frac{1}{2}u^2+\frac{1}{2}\phi*(\frac{3}{2}u^2))=0,\quad t>0, \; \, x\in\R,
\end{equation}
where $ \phi(x) = e^{-|x|} $ and $``* "$ stands for convolution with respect to the spatial variable $x\in\R$. Peakons are weak solutions of these conservation laws and are true solitons that interact via elastic collisions respectively under the CH dynamics and the DP dynamics.  Moreover, as a fundamental qualitative property in nonlinear dynamics, the orbital stability of peakons of the CH and DP equation has been verified \cite {C-S, L-L}. Relevant stability results for waves approximating peakons are also available \cite {CM}.
However, the DP equation distinguishes from the
CH equation substantially, mainly in the corresponding conservation laws. Much more sophisticated arguments are used in \cite{L-L} to overcome the much weaker $L^2$ conservation law presented in the DP equation.
Another novel feature of the DP equation is that for $k=0$,  not only does it have peaked solitons \cite{CH, D-H-H} of the form  $ u(t,x)=ce^{-| x-ct|},  \, c \in \mathbb {R}$,  it also admits shock peakons \cite{E-L-Y, Lu} of the form
$$ u(t,x)=-\frac{1}{t+a }\text{sgn}(x)e^{-| x|},\, a >0.
$$ It is not  clear if such a discontinuous solution is stable or not in proper settings.

In the case of non-vanishing linear dispersion ($k \neq 0$), the existence and stability of localized smooth solitary waves of the CH  equation \eqref{CH} are well understood by now \cite {C-H-H, C-S3}, while the DP equation case has been barely explored so far and thus the subject of this paper. The goal of this paper is to establish existence and spectral stability results of smooth solitons for the DP equation \eqref{DP2k}.

We start with a rigorous definition of solitary waves, i.e., solitons.  Firstly, a solution of the DP equation $u(t,x)$ is a \textit{traveling wave} if there exist a real number $c$ and a scalar function $\phi: \R\to\R$ such that
\[
u(t,x)=\phi(x-ct).
\]
Moreover, a traveling wave of the DP equation $\phi(x-ct)$ is a \textit{solitary wave} if there exists $\xi_0\in\R$ such that
\begin{itemize}
\item $\max\limits_{\xi\in\R}\phi(\xi)=\phi(\xi_0)$ and $\lim\limits_{\xi\to\pm\infty}\phi(\xi)=0$.
\item $\phi$ is strictly increasing on $(-\infty,\xi_0)$ and strictly decreasing on $(\xi_0,\infty)$.
\end{itemize}
We give the existence result below,  refer to Section \ref{s:2} for further reading and move on directly to the discussion of the essential topic--the stability issue.
\begin{theorem}[existence]\label{profile}
Under the physical condition $c>2k>0$, there exists, up to translations, a unique $c-$speed solitary wave $\phi(\xi; c)$ with its maximum height $$\frac{c-2k}{4}<\phi_{max}\triangleq \max\limits_{\xi\in\R} \{\phi\}<c-2k.$$ In addition, the function $\phi(\xi; c)$ could be taken even and strictly monotonically increases from $0$ to $\phi_{max}$ for negative values of $\xi$.
\end{theorem}

Thanks to the translation invariance of the equation, for any given solitary wave $\phi(\xi; c)$, its spatial translation generates a family of solutions, called \textit{the orbit of the solitary wave}, denoted as
\[
O_c=\{\phi(\cdot+x_0; c) \mid \; x_0\in\R\}.
\]
Moreover, a solitary-wave solution $\phi$ of the DP  equation is called {\it orbitally stable} if a wave starting close to the solitary wave $ \phi$ remains
close to the orbit of the solitary wave up to the existence time. A generic feature of nonlinear dispersive equations is that their solutions usually tend to be oscillations that, as time evolves, spread out spatially in a significantly nonlinear and complicated way. When it comes to solitary waves, one would naively expect that a small perturbation of a solitary wave would at least yield another one with a different speed and phase shift, if not more complicated, which makes the stability of solitary waves counter-intuitive and thus fascinating.

Another weaker form of stability is called {\it spectral stability}. A solitary wave is called  {\it spectrally stable} if the corresponding linearized equation admits no exponentially unstable solution.

The study of stability is essentially based on the Hamiltonian structure of the DP equation applying some general index counting theorem from \cite{LZ}. Actually, the DP equation \eqref{DP2k} in terms of $u$, that is,
\begin{equation}
u_t-u_{xxt}+2ku_x-3u_xu_{xx}-uu_{xxx}+4uu_x=0,
\end{equation}
can be written as an infinite dimensional Hamiltonian PDE, that is,
\begin{equation}\label{DP}
u_t=J\frac{\delta H}{\delta u}(u),
\end{equation}
where $$J\triangleq\partial_x(4-\partial_x^2)(1-\partial_x^2)^{-1},\quad H(u)\triangleq-\frac{1}{6}\int \left ( u^3+6k \left ((4-\partial_x^2)^{-\frac{1}{2}}u \right )^2 \right ) \, dx. $$
It  is observed  that some relevant conservation laws of the DP equation \eqref{DP2k} are generically weaker than those of the CH equation \eqref{CH}. More specifically, there are at least three relevant conservation laws of \eqref{DP2k} in study of stability---the conservation of momentum $M(u)$, the Hamiltonian $H(u)$, the conserved quantity $S(u)$ arising from the translation symmetry, respectively taking the following forms.
\begin{equation*}
M(u)=\int_\R(1-\partial_x^2)u\,dx, \quad\quad H(u)=-\frac{1}{6}\int_\R \left (u^3+6ku\cdot
(4-\partial_x^2)^{-1}u \right ) \,dx,
\end{equation*}
\begin{equation}\label{E3DP}
S(u)=\frac{1}{2}\int_\R(1-\partial_x^2)(4-\partial_x^2)^{-1}u\cdot u\,dx,
\end{equation}
while the corresponding ones of the CH equation  \eqref{CH} are the following,
\begin{equation*}
M(u)=\int_\R(1-\partial_x^2)u\,dx, \quad\quad \widetilde{H}(u)=\int_\R \left (u^3+uu_x^2+2ku^2 \right ) \,dx,
\end{equation*}
\begin{equation}\label{E3CH}
\widetilde{S}(u)=\int_\R \left ( u^2+u_x^2 \right ) \,dx.
\end{equation}

\begin{re}
Using Kato's theorem \cite {K},   it is known \cite {Y1} that  if  initial profiles $  u_0 \in  H^s(\R) $ with $s > \frac{3}{2}, $ \eqref{DP2k} has a unique solution in $ C([0, T); H^s(\R)) $ for some $ T > 0 $ with $ M,H $ and $S$ all  conserved.  Moreover, the
only way that a classical solution of equation \eqref{DP2k} fails to exist for all time is that the
wave breaks. Some solutions of \eqref{DP2k} are defined globally in time (e.g. the smooth
solitary waves) while other waves break in finite time \cite {L-Y}.
\end{re}
\begin{re} While it is straightforward to verify that $M$ and $H$ are conserved quantities, the verification of the conservation of $S$ under the flow is relatively nontrivial. In fact, the conservation of $S(u)$ holds as long as the Hamiltonian density at spatial infinity equal to zero. More specifically, for any solution $u(t,x)$ to the DP equation with initial condition $u(0, \cdot)\in H^s(\R)$ with $s>3/2$, the solution $u(t,x)$ is continuous in $x$ with $\lim\limits_{x\to\pm\infty}u(t,x)=0$ and
\begin{equation*}
\begin{aligned}
\frac{d S}{d t}&=((1-\partial_x^2)(4-\partial_x^2)^{-1}u, u_t)=((1-\partial_x^2)(4-\partial_x^2)^{-1}u, J\frac{\delta H}{\delta u}(u))\\
&=-(\partial_x u, \frac{\delta H}{\delta u}(u))=\int_{\R} \partial_x h(u(t,x))dx\\
&=h(u(t,\infty))-h(u(t,-\infty))=h(0)-h(0)=0,
\end{aligned}
\end{equation*}
where $h(u)=-\frac{1}{6}\left[u^3+6k \left ((4-\partial_x^2)^{-\frac{1}{2}}u \right )^2\right]$ is the Hamiltonian density.
\end{re}
In particular, one can see that the conservation law $S$ for the DP equation is equivalent to $\Vert u\Vert_{L^2}^2$.  In fact, by the Fourier transform, we have
\begin{equation*}
S(u)= \frac{1}{2} \int_\R (1 - \partial_x^2) u (4-\partial_x^2)^{-1}udx= \frac{1}{2} \int_\R\frac{1+\xi^2}{4+\xi^2}\vert\hat u
(\xi)\vert^2d\xi\sim\Vert\hat u\Vert_{L^2}^2=\Vert u\Vert_{L^2}^2.
\end{equation*}

Due to such a weaker conservation law $S$ for the DP equation, compared with $\widetilde{S}$ of the CH case, we can only expect spectral (or orbital) stability of solitons in the sense of the $L^2$ norm, which makes the study of the stability of smooth DP solitons much more subtle.

In fact, taking advantage of the fact that the conserved energy $\widetilde{S}$ in \eqref{E3CH} of the CH equation is $H^1$ norm of the solution and fixed sign of the momentum density, the variational framework by Grillakis, et.al. \cite{GSS} can be successfully applied without too much trouble to obtain orbital stability of smooth CH solitons \cite{CM}. More specifically,
\begin{itemize}
\item  According to the conservation law of momentum $M$, the CH skew symmetric operator
\[
J_{CH}=-\partial_x(1-\partial_x^2)^{-1}
\]
is bounded and invertible when restricted to the zero-average co-dimensional one subspace.

\item By the Liouville substitution,  the linearized operator
\[
\mathcal{L}_{CH}=-\partial_x((2c-2\phi)\partial_x)-6\phi+2\phi''+2(c-k)
\]
with respect to the soliton $ \phi $, defined on the space $H^2(\R)$, is transformed into a regular self adjoint Sturm-Liouville operator, which is, as one readily sees, a relatively compact perturbation of a second order differential operator with constant coefficients. The required spectral properties of $\mathcal{L}_{CH}$ then follows directly from the Sturm-Louville theory.

\item  The corresponding convexity condition is easily verified in the CH soliton case which takes big advantage of the simple form of the conservation law $\widetilde{E}_3$.

\item  The strong $L^2$ coercivity on restrained space can be lifted to $H^1$, so as to control the remaining nonlinear part and to obtain orbital stability.

\end{itemize}
As for the orbital stability of smooth DP solitons, there are several obstacles to tackle.
\begin{itemize}
\item  The corresponding  DP skew symmetric operator
\[
J_{DP}=-\partial_x(4-\partial_x^2)(1-\partial_x^2)^{-1}
\]
is not bounded invertible. This obstacle is mild, since the generator of the translation symmetry is $\partial_x$, annihilating the unbounded part $\partial_x^{-1}$ in the pseudo inversion $J_{DP}^{-1}$ and making $J_{DP}^{-1}\partial_x$ bounded invertible, as in the KdV case.
\item The corresponding linearized operator
\[
\mathcal{L}_{DP}=(c-2k-c\partial_\xi^2)(4-\partial_\xi^2)^{-1}-\phi
\]
fails to directly transform into a regular self-adjoint Sturm-Liouville type  operator, so the study of its spectral properties becomes highly nontrivial. While the essential spectrum and the simplicity of the eigenvalue $0$ of the operator $\mathcal{L}_{DP}$ can be readily obtained via a neat and brief functional analysis argument, the analysis on negative eigenvalues roots deep in dynamical systems and is technical. 
%the frame work of \textcolor{red}{Grillakis et al.} \cite{GSS} can not be applied directly. \textcolor{red}{As a remedy}, we take advantage of the newly developed frame work of Lin and Zeng \cite{LZ} for the  spectral stability to show there exists no exponentially unstable solution of \eqref{hamilton}. For the nonlinear stability, inspired by \cite{LZ}, we need to have a \textcolor{red}{solid} understanding of the numerical range of the corresponding linearized operator $ \mathcal{L}_{DP} $ for the DP equation as well as \textcolor{red}{the} verification of a convexity property; \textcolor{red}{see Section \ref{s:4} for details}.
%
%\item The study of positive range of $(\mathcal{L}_{DP} \, \cdot,  \cdot)$ becomes nontrivial because $\mathcal{L}_{DP}$ is not a relatively compact perturbation of
%    \[
%    \mathcal{L}^\infty:=(c-2k-c\partial_\xi^2)(4-\partial_\xi^2)^{-1},
%    \]
%    which acts on $L^2(\R)$ instead of $H^2(\R)$.
%
%\item The study of negative range of $(\mathcal{L}_{DP}\, \cdot , \cdot)$ is related to an eigenvalue problem in non-standard form.

\item The verification of the convexity of the Lagrangian evaluated at solitary wave profiles with respect to the wave speed $ c $ is also highly nontrivial and relies substantially on the special structure of the DP equation.

\item  The strong $L^2$ coercivity (to be proved for the DP soliton in this paper) on restrained space can NOT be lifted to $H^1$, essentially because of the weaker conservation law $S$ of DP equation which leads to consideration of an non-differential operator $\mathcal{L}_{DP}$.
\end{itemize}
We succeed in tackling the first three listed obstacles and have the following main result.

\begin{theorem}\label{stability}
Let $u^c(t,x)$ be the  solitary-wave solution of \eqref{DP2k} with its traveling speed $c > 2k$. Such a solitary-wave $u^c(t,x)$ is spectrally stable in $L^2(\mathbb{R})$.
\end{theorem}

\begin{re}
As explained above that the strong $L^2$ coercivity can not be lifted to $H^1$ because of the special form of the linear operator $\mathcal{L}_{DP}$. One can actually prove that the strong  $H^1$ coercivity fails on corresponding restrained space by performing parallel spectral analysis of $\mathcal{L}_{DP}$ viewed as an operator on $H^1$ which admits positive continuous spectrum touching $0$.
\end{re}
\begin{re}
While the term $\int u^3\mathrm{d}x$ can be bounded by the $H^1$-norm of $u$, it generically can not be bounded by the $L^2$-norm of $u$. As a result, we note that, 
in order to obtain orbital stability results of DP smooth solitons, a remedy is to establish the stability in the $L^2$ and $L^\infty$ norm simultaneously. For peaked soliton $ \varphi = c e^{-|x-ct|} $ with $ k = 0, $ the control of the $L^\infty$ norm is reduced to bound the pointwise distance between the peakon's maximum and the maximum height of the perturbed profile, which is related to the the perturbation in $ L^2$-norm. This could be observed from the following formula derived in \cite {L-L},
\[
S(u) - S(\varphi) = S(u - \varphi(\cdot- \xi)) + 2 c \left(
v_{u}(\xi) - v_{\varphi}(0) \right),
\]
where $v_{u} = (4 - \partial_{x}^{2})^{-1} u $ and $ \xi \in \mathbb{R}. $ However, it seems not easy  to control the $L^\infty$ norm in the case of smooth solitons, leaving the orbital stability of these smooth solitons as an open problem.
\end{re}

The remainder of the paper is organized as follows. In Section \ref{s:2} we establish the existence  and properties of the smooth solitary-wave solutions (Theorem \ref{profile}). In Section \ref{s:3}, we study the spectrum of the corresponding linear operator of the second-order variational derivative of the Hamiltonian. In Section \ref{s:4} we give the proof of the spectral stability result (Theorem \ref{stability}), %{\color{blue} and explain the main obstacle in getting the orbital stability. }

\bigskip

\section{Smooth Degasperis-Procesi Solitons}\label{s:2}

In this section, we study smooth solitary wave of \eqref{DP},
\[
u_t=J\frac{\delta H}{\delta u}(u),
\]
where $$J=\partial_x(4-\partial_x^2)(1-\partial_x^2)^{-1},\quad H(u)=-\frac{1}{6}\int \left ( u^3+6k \left ((4-\partial_x^2)^{-\frac{1}{2}}u \right )^2 \right ) \, dx. $$
Changing  the $(t,x)$ coordinates into the traveling frame $(t,\xi)$  with $\xi\triangleq x-ct$ and slightly abusing the notation by denoting $u(t, \xi)\triangleq u(t, x-ct)$,  the equation \eqref{DP} is now written as
\begin{equation}\label{travelingDP}
u_t=J\frac{\delta H}{\delta u}(\phi)+cu_\xi=J(\frac{\delta H}{\delta u}(u)+c\frac{\delta S}{\delta u}(u)),
\end{equation}
where we recall $$S(u)=\frac{1}{2}\int(1-\partial_\xi^2)(4-\partial_\xi^2)^{-1}u\cdot ud\xi.$$
Introducing the Lagrangian
\begin{equation}\label{lag}
Q_c(u)\triangleq H(u)+ c S(u),
\end{equation}
the solitary wave with speed $c>0$, denoted as $\phi(\xi; c)$, is a steady state of \eqref{travelingDP} and a critical point of the Lagrangian, namely,
\begin{equation}\label{sttDP}
\frac{\delta Q_c}{\delta u}(\phi)=\frac{\delta H}{\delta u}(\phi)+c\frac{\delta S}{\delta u}(\phi)=0.
\end{equation}

We now prove Theorem \ref{profile}.
%
%\begin{theorem}\label{profile}
%Under the physical condition $c>2k$, there exists a unique $c-$speed solitary wave $\phi(\xi; c)$ with \textcolor{red}{its maximum height $$\frac{c-2k}{4}<\phi_{max}\triangleq \max\limits_\xi \{\phi\}<c-2k.$$ In addition, the function $\phi(\xi; c)$} is even and strictly monotonically increases from $0$ to $\phi_{max}$ for negative values of $\xi$.
%\end{theorem}
\begin{proof}[Proof of Theorem \ref{profile}]
The stationary traveling DP equation \eqref{sttDP} is equivalent to
\begin{equation}\label{traveling}
-[\frac{1}{2}\phi^2+(4-\partial_\xi^2)^{-1}2k\phi]+c(1-\partial_\xi^2)(4-\partial_\xi^2)^{-1}\phi=0.
\end{equation}
Applying $4-\partial_\xi^2$, we get
\[
-2\phi^2+\phi\phi_{\xi\xi}+\phi_\xi^2 -2k\phi+c(\phi-\phi_{\xi\xi})=0,
\]
In terms of a system of first order ODEs, we have
\begin{equation}
 \left\{\begin{array}{rl}
\phi_\xi&=\psi\\
(c-\phi)\psi_\xi&=(c-2k)\phi-2\phi^2+\psi^2,\\
 \end{array}
 \right.
\end{equation}
which has a first integral
\[
\Phi(\phi,\psi)=\phi^2(\frac{1}{2}\phi^2-c\phi+\frac{2}{3}k\phi+\frac{1}{2}c^2-kc)-\frac{1}{2}(c-\phi)^2\psi^2.
\]
A solitary wave of \eqref{travelingDP} corresponds to the $\phi$ entry of the connected component of the level curve
\[
\Phi(\phi,\psi)=\Phi(0,0)=0,
\]
which connects to the origin.
Any point $(\phi, \psi)$ on the level curve $\Phi(\phi,\psi)=0$ satisfies
\begin{equation}\label{firstintegral}
\phi^2(\frac{1}{2}\phi^2-c\phi+\frac{2}{3}k\phi+\frac{1}{2}c^2-kc)=\frac{1}{2}(c-\phi)^2\phi_\xi^2,
\end{equation}
and thus the quadratic polynomial, $P(\phi)\triangleq
\frac{1}{2}\phi^2-c\phi+\frac{2}{3}k\phi+\frac{1}{2}c^2-kc$, is nonnegative. Noting that, given $c>2k$, the polynomial $P(\phi)$ admits two distinctive positive real roots, denoted as $\phi_\pm$, admitting the following expressions
\[
\phi_\pm=c-\frac{2}{3}k\pm\sqrt{\frac{2}{9}k(3c+2k)}>0,
\]
we conclude that the level curve $\Phi(\phi,\psi)=0$ has two connected components, lying respectively within the region $\phi\geq\phi_+$ and $\phi\leq\phi_-$. The solitary wave profile $\phi(\xi; c)$, as part of the level curve $\Phi(\phi,\psi)=0$, is a subset of the connected component within $\phi\leq\phi_-$. More specifically, we readily see from the geometry of the level curve \eqref{firstintegral}, which is symmetric about $\psi$, that the solitary wave profile $\phi(\xi;c)$ is even with respect to $\xi$ and is strictly increasing on $\xi\in(-\infty,0]$ with its global maximum $\phi_{max}$ obtained at $\xi=0$; that is,
\[
\phi_{max}=\phi_-=c-\frac{2}{3}k-\sqrt{\frac{2}{9}k(3c+2k)}.
\]
Given that $c>2k>0$, it is also straightforward to see that $\phi_{max}\in(\frac{c-2k}{4}, c-2k)$, which concludes the proof of the theorem.

%Setting $\phi_\xi=0$ to find the maximum value of $\phi_{max}$ of the solitary wave $\phi(\xi)$:
%\begin{equation}
%\phi_{max}=c-\frac{2}{3}k-\sqrt{\frac{2}{9}k(3c+2k)},
%\end{equation}
%where we have taken the smaller root of a quadratic equation, for the reason that the larger one corresponds to the unbounded part of the level curve $\Phi(\phi,\psi)=0$. \textcolor{red}{Also, a sufficient condition for $\phi_{max}>0$ is that $$c>2k,$$ under which,  it is straightforward to see that  $$\frac{c-2k}{4}<\phi_{max}<c-2k.$$
%Moreover, it follows easily from the geometry of the level curve \eqref{firstintegral}, }which is symmetric about $\psi$, that the function $\phi(\xi)$ is even and strictly monotonically increases from $0$ to $\phi_{max}$ \textcolor{red}{for $\xi\in(-\infty, 0]$, which concludes} the proof of Theorem \ref {profile}.
\end{proof}

\begin{re}
It is known that in the limit $ k \to 0, $  the solitary waves of the CH equation \eqref{CH} with maximal elevation at $ x = 0 $ converge
uniformly on every compact subset of $\mathbb {R} $  to the peakon $ \varphi (x) =  c e^{-|x|}  $  \cite{LiOl}. For the DP equation \eqref{DP2k}, a similar convergence result of the smooth solitary  waves to peakon can also be obtained by studying the following ODE:
\[
\frac{u_\xi(\xi)}{u(\xi)}=\frac{\sqrt{u^2-2(c-\frac{2}{3}k)u+c^2-2ck}}{c-u}.
\]

\end{re}

\bigskip

\section{Spectral Analysis}\label{s:3}

In this  section, we study the spectrum of the corresponding linear operator of the second-order variational derivative of the Lagrangian, which is critical to the stability of smooth solitary waves. From now on, we simply write $\phi$ for the solitary wave profile $\phi(\xi;c)$ unless specified.

Consider under the traveling frame $(t,\xi)$ the linearization of \eqref{travelingDP} along the soliton $\phi$,
\begin{equation}\label{linearization}
v_t=JL_cv,
\end{equation}
where $v\in L^2(\R)$, and
\[
L_c=\frac{\delta^2 Q_c}{\delta u^2}(\phi)
=-\phi-2k(4-\partial_\xi^2)^{-1}+c(1-\partial_\xi^2)(4-\partial_\xi^2)^{-1}
%&=&(c-2k-c\partial_\xi^2)(4-\partial_\xi^2)^{-1}-\phi\\
=c-\phi-(3c+2k)(4-\partial_\xi^2)^{-1}.
\]
It is straightforward to see that 
$L_c:  L^2(\R)\to L^2(\R)$ is a well-defined, self-adjoint, bounded linear operator. Moreover, we have the following spectral theorem about the operator $L_c$.
\begin{theorem}\label{spectrum}
The spectrum of the operator $L_c$, denoted as $\sigma(L_c)$, admits the following properties.
\begin{enumerate}
\item The spectrum set $\sigma(L_c)$ lies on the real line; that is, $\sigma(L_c)\subset \R$.
\item $0$ is a simple isolated eigenvalue of $L_c$ with $\partial_\xi \phi (\xi)$ as its eigenfunction.
\item On the negative axis $(-\infty, 0)$, the spectrum set $\sigma(L_c)$ admits nothing but only one simple eigenvalue, denoted as $\lambda_*$, with its corresponding normalized eigenfunction, denoted as $\phi_*$.
\item The set of essential spectrum $\sigma_{ess}(L_c)$ lies on the positive real axis, admitting a positive distance to the origin.
\end{enumerate} 
%is a subset of the real line and $Ker\{L_c\}=\{\partial_\xi \phi\}$.
%There is exactly one simple eigenvalue of $L_c$ on the negative real axis with the corresponding \textcolor{red}{normalized} eigenfunction which we denote by $\phi_*$. The essential spectrum $\sigma_{ess}(L_c)$ is a subset of the positive real line and has a positive distance to the origin.
\end{theorem}

\begin {proof}  Statement (1) is straightforward, due to the fact that $L_c$ is self-adjoint. We now prove statement (4) first.
Denoting $\lambda_0\triangleq \min\{\frac{c-2k}{4}, \frac{c-\phi_{max}}{2}\}>0$, we consider the eigenvalue problem
\[ L_cv=\lambda v, \quad \mbox{ for } \lambda\in(-\infty,\lambda_0), \mbox{ and }v\in L^2(\R).
\]
Introducing the notation
 \[
 p\triangleq(4-\partial_\xi^2)^{-1}v, \qquad A(\xi, \lambda)\triangleq\frac{c-2k-4\phi(\xi)-4\lambda}{c-\phi(\xi)-\lambda},
 \]
 the above eigenvalue problem is equivalent to
\begin{equation}\label{eigen}
L_\lambda p\triangleq p_{\xi\xi}-A(\xi,\lambda)p=0,\quad q\in H^2(\R).
\end{equation}
Note that $\phi (\xi)< c-2k<c$ for any $\xi\in(-\infty,\infty)$ and $\lambda<\lambda_0$, so the coefficient $A(\xi, \lambda)$ is well-defined in the sense that its numerator $c-\phi -\lambda$ is always positive. Moreover, for any given $\lambda\in(-\infty, \lambda_0)$, the constant $\frac{c-2k-4\lambda}{c-\lambda}>0$, yielding that the operator
\[
\begin{aligned}
L_\lambda^\infty: &H^2(\R) &\longrightarrow &L^2(\R),\\
&\quad p & \longmapsto & -\partial_\xi^2p+\frac{c-2k-4\lambda}{c-\lambda}p,
\end{aligned}
\]
is bounded invertible and thus Fredholm with index $0$. As a result, for any $\lambda\in(-\infty, \lambda_0)$, the operator $L_\lambda$ is Fredholm with index $0$ for that it differs by a compact perturbation from the operator $L_\lambda^\infty$, which concludes the proof of statement (4) and indicates that any $\lambda\in(-\infty, \lambda_0)$ is either in the resolvent of $L_c$ or an eigenvalue of $L_c$. As a matter of fact, from general properties of Sturm-Liouville operators \cite{K-P}, if $\lambda\in(-\infty, \lambda_0)$ is an eigenvalue of $L_\lambda$ (and thus of $L_c$), then it must be a simple eigenvalue. 

Statement (2) is basically a consequence of the above argument. Taking $\lambda=0$ and noting that $(4-\partial_\xi^2)^{-1}\partial_\xi \phi $ is a solution of the eigenvalue problem $L_0 p=0$ by differentiate the traveling wave equation with respect to $\xi$, we conclude that statement (2) is true.

We consider last the statement (3). The proof takes advantage of properties of the coefficient $A(\xi,\lambda)$, which is even in $\xi$ and strictly increasing on the interval $\xi\in [0,\infty)$. Moreover, since $\phi (\xi)$ decays exponentially as $\xi$ approaches $\pm\infty$, and $\lambda\leq0$, the coefficient of $q$,  $-A(\xi,\lambda)$, is always negative for non-positive values of $\lambda$ and large enough $\vert\xi\vert$.
It follows that there always exist unique Jost solutions $J^s(\xi,\lambda)$ and $J^u(\xi,\lambda)$, up to multiplication of a constant, which approaches $0$ as $\xi\to\pm\infty$ respectively. In fact, $J^{s\backslash u}(\xi,\lambda)$ is asymptotic  to the limit solution
\[
J_\infty^{s\backslash u}(\xi,\lambda)\triangleq e^{\mp\sqrt{A(\infty, \lambda)}\xi}, \qquad A(\infty,\lambda)\triangleq \lim\limits_{\xi\to\pm\infty}A(\xi,\lambda)=\frac{c-2k-4\lambda}{c-\lambda};
\]
that is,
\[
J^{s\backslash u}(\xi,\lambda)\simeq e^{\mp\sqrt{\frac{c-2k-4\lambda}{c-\lambda}}\xi},\quad \mbox{ as }\xi\to\pm\infty.
\]

{\it The eigenvalue problem \eqref{eigen} is really a shooting one. $\lambda$ is an eigenvalue of \eqref{eigen} if and only if the two vectors $$(J^s(0,\lambda),\frac{d}{d\xi}J^s(0,\lambda)) \mbox{ and } (J^u(0,\lambda),\frac{d}{d\xi}J^u(0,\lambda))$$ are parallel to each other.}

\smallskip

While it is possible to compute the corresponding Evans function to locate the eigenvalues of \eqref{eigen}, we prefer to give a more geometric proof which takes advantage of the special structure of solitary wave, specifically the evenness of the solitary wave profile $\phi (\xi)$.

Under the polar coordinates change
\begin{equation*}
 \left\{\begin{array}{rl}
q&=\rho  \cos\theta\\
q_\xi&=\rho  \sin\theta,\\
 \end{array}
 \right.
\end{equation*}
 the equation \eqref{eigen} becomes
\begin{equation}\label{angle}
\theta_\xi=A(\xi,\lambda)cos^2\theta-sin^2\theta.
\end{equation}
and an equation of $\rho$ which is slaved to the $\theta$ equation and thus omitted.
{\it An eigenvalue $\lambda_{ei}$ is such that there is a solution $\theta_{ei}(\xi,\lambda_{ei})$ of \eqref{angle} which approaches
\[
\theta^{-\infty}(\lambda)\triangleq \arctan(\lim_{\xi\to\pm\infty}\frac{\frac{d}{d\xi}J^s(\xi,\lambda)}{J^s(\xi,
\lambda)})=\arctan\sqrt{A(\infty,\lambda)},
\]
as $\xi\to-\infty$ and approaches
\[
\theta_k^{\infty}(\lambda)\triangleq\arctan(\lim_{\xi\to\infty}\frac{\frac{d}{d\xi}J^s(\xi,\lambda)}{J^s(\xi,
\lambda)})+k\pi=k\pi-\arctan\sqrt{A(\infty,\lambda)}, \qquad k\in \mathbb{Z},
\]
as $\xi\to\infty$,
 where, to prevent multiple counting of eigenfunctions, we intentionally set $\theta^{-\infty}$ fixed, in the sense that $\theta^{-\infty}$ does not depend on $k$ while $\theta^\infty_k$ does, so that the shooting problem has a fixed ``start point" at $-\infty$, that is, $\theta^{-\infty}$, and infinite many valid choices of ``end points" at $\xi=+\infty$, that is, $\{\theta^\infty_k\}_{k\in\mathbb{Z}}$.}

It is observed that the coefficient $A(\xi,\lambda)$, as a function of $\xi$ for fixed $\lambda$, admits the following dichotomy. 
\begin{itemize}
\item  For any fixed $\lambda\in(-\infty, \lambda_0)$ with $\lambda_1\triangleq \frac{c-2k}{4}-\phi_{max}$,
\[
A(\xi, \lambda)>0,\quad \mbox{ for all }\xi;
\]
\item  For any fixed $\lambda\in[\lambda_0, 0]$, there exists $\bar\xi(\lambda)\geq0$ such that
\[
A(\xi, \lambda)\begin{cases} \geq 0, \quad \vert\xi\vert\geq\bar\xi(\lambda), \\ \leq0,\quad \vert\xi\vert\leq\bar\xi(\lambda).\end{cases}
\]
\end{itemize}

\bigskip

There can not be any solution of the angle equation \eqref{angle} approaching $\theta^{-\infty}(\lambda)$, $\theta_k^{\infty}(\lambda)$ as $\xi\to\pm\infty$. Indeed, the angle equation \eqref{angle} has four $(\xi,\lambda)$-dependent pseudo-equilibrium on the unite circle:
\[\Theta_1^\pm(\xi,\lambda)\triangleq\pm\arctan\sqrt{A(\xi,\lambda)},
\quad \Theta_2^\pm(\xi,\lambda)\triangleq\pm\arctan\sqrt{A(\xi,\lambda)}-\pi. \]
$\Theta_1^+$ and $\Theta_2^+$ attract nearby points, while $\Theta_1^-$ and $\Theta_2^-$ repel nearby points. For conveniences, we also introduce the notation $B_k\triangleq-\frac{k\pi}{2}$ for $k\in \mathbb{N}$.  \\

{\bf Case 1.} Any $\lambda\in(-\infty, \lambda_0)$ is not an eigenvalue since the first quadrant, $\theta\in[0,\pi/2]$, is forward-invariant, which indicates that if $\theta(\xi,\lambda)$ starts at $\xi=-\infty$ from $\theta^{-\infty}(\lambda)=\Theta_1^+(-\infty,\lambda)\in[0,\pi/2]$, it will not leave $[0,\pi/2]$ and thus never reaches any $\theta_k^\infty$, which lies in the second and fourth quadrant. More specifically, as $\xi$ increases from $-\infty\to0$, $\Theta_1^+$ and $\Theta_1^-$ move  respectively clockwise and counterclockwise towards, but never reach, $B_0$, and as $\xi$ then increases from $0\to\infty$, the whole process is reversed; that is, $\Theta_1^+$ and $\Theta_1^-$ return to their starting positions at $-\infty$ in exactly the same speed but opposite directions. Essentially, the movement of $\Theta_1^+$ and $\Theta_1^-$ is a reflection about the line $\{\theta=0\}$ of each other because of the evenness of $\phi (\xi)$.
The movement of $\Theta_2^+$ and $\Theta_2^-$ are exactly the same as that of $\Theta_1^+$ and $\Theta_1^-$ in obvious sense.

As a matter of fact, we can show that, for $\xi\in[0,\infty)$, the angle evolution, $\theta^s(\xi,\lambda)$, of the stable Just solution $J^s(\xi,\lambda)$, is strictly decreasing, and shadows the unstable pseudo-equilibria $\Theta_{1\backslash 2}^-$, modulo $2\pi$. More specifically, $\theta^s(\xi,\lambda)$ can be viewed as the solution to the angle equation \eqref{angle} with the limit boundary condition $\theta^s(+\infty,\lambda)=\Theta_{1\backslash 2}^-(+\infty,\lambda)\mod(2\pi)$.  For conveniences, we simple set $\theta^s(+\infty,\lambda)=\Theta_{2}^-(+\infty,\lambda)$. Noting that the intervals
\[
(\Theta_2^-(+\infty, \lambda), \theta_1^-(0,\lambda)-2\pi) \text{ and } (\Theta_2^-(0,\lambda), \Theta_1^-(+\infty, \lambda)),
\]
are forward invariant, we conclude that the angle of the stable Jost function, $\theta^s(\xi,\lambda)$, is trapped within the interval $(\Theta_2^-(+\infty, \lambda), \Theta_2^-(0,\lambda))$; that is,
\[
\{\theta^s(\xi,\lambda)\mid \xi\in[0,\infty)\}\subset(\Theta_2^-(+\infty, \lambda), \Theta_2^-(0,\lambda)).
\]
Moreover, we claim that
\[
\theta^s(\xi,\lambda)<\Theta_2^-(\xi,\lambda), \text{ for any  }\xi\in[0,\infty),
\]
essentially due to the fact that the unstable pseudo-equilibrium $\Theta_2^-(\xi,\lambda)$ is strictly decreasing with respect to $\xi\in[0,\infty)$. We prove by contradiction. If this is not true, we have the following two scenarios.
\begin{enumerate}
\item If there exists $\xi_0\in[0,\infty)$ such that $\theta^s(\xi_0,\lambda)>\Theta_2^-(\xi_0,\lambda)$, then it is straightforward to see that the interval $(\Theta_2^-(\xi_0,\lambda), -\pi)$ is forward invariant for $\xi\in[\xi_0, +\infty)$. As a result, we have
\[
\Theta_2^-(+\infty, \lambda)=\lim\limits_{\xi\to\infty}\theta^s(\xi,\lambda)\in[\Theta_2^-(\xi_0,\lambda), -\pi]\not\ni \Theta_2^-(+\infty, \lambda),
\]
which is a contradiction.
\item If there exists $\xi_0\in[0,\infty)$ such that $\theta^s(\xi_0,\lambda)=\Theta_2^-(\xi_0,\lambda)$, then we claim that the interval $[\Theta_2^-(\xi_0,\lambda), -\pi)$ is forward invariant for $\xi\in[\xi_0, +\infty)$ and thus a contradiction follows as in the previous case. The subtle part of the forward invariance lies at the inclusion of the left end point $\Theta_2^-(\xi_0,\lambda)$, due to the fact that, if $\theta^s(\xi_0,\lambda)=\Theta_2^-(\xi_0,\lambda)$, then
\[
\partial_\xi \theta^s(\xi_0,\lambda)=0>\partial_\xi \Theta_2^-(\xi_0,\lambda),
\]
letting $\theta^s$ fall behind $\Theta_2^-$ and converges towards $\Theta_2^+$ as $\xi$ goes to infinity from $\xi_0$.
\end{enumerate}
According to the evenness of $A$ with respect to $\xi$, similar conclusions can be drawn for the angle evolution, $\theta^u(\xi,\lambda)$, of the unstable Jost solution $J^u(\xi,\lambda)$.
%If $\theta(\xi,\lambda)$ starts at $\xi=-\infty$ from $\theta^{-\infty}(\lambda)$, which is also the same as $\Theta_1^+(-\infty,\lambda)$, then \textcolor{blue}{it follows exactly the orbit of $\Theta_1^+(\xi,\lambda)$} \textcolor{red}{Comment: this is not true, I just recovered my proof and will add it later today.}
\smallskip

{\bf Case 2. } In the interval $[\lambda_0, 0]$, we claim that there are only two eigenvalues; that is, $\lambda=0$ and $\lambda=\lambda_*\in(\lambda_0, 0)$. It is straightforward to see that
\begin{center}
$\lambda$ is an eigenvalue if and only if $\theta^s(0,\lambda)=\theta^u(0,\lambda)$,
\end{center}
which, thanks to the evenness of $A(\xi,\lambda)$ with respect to $\xi$, is equivalent to the following statement.
\[
\lambda \text{ is an eigenvalue if and only if } \theta^u(0,\lambda)=B_k, \text{ for some }k\in\mathbb{N}.
\]
Noting that for any $\lambda\in(-\infty,0]$,
\[
\partial_\lambda A(\xi,\lambda)=-\frac{3c+2k}{(c-\phi(\xi)-\lambda)^2}<0, \quad \partial_\lambda \theta^{-\infty}(\lambda)=\frac{\partial_\lambda A(\infty,\lambda)}{2\sqrt{A(\infty,\lambda)}(1+A(\infty, \lambda))}<0,
\]
we conclude that
\[
\begin{aligned}
\partial_\lambda \theta^u(0,\lambda)&=\partial_\lambda \big( \theta^u(0,\lambda) - \theta^{-\infty}(\lambda)\big) + \partial_\lambda \theta^{-\infty}(\lambda)\\
&=\partial_\lambda \big( \int_{-\infty}^0 (A(\xi,\lambda)\cos^2\theta-\sin^2\theta)d\xi\big) + \partial_\lambda \theta^{-\infty}(\lambda)\\
&=\int_{-\infty}^0(\partial_\lambda A(\xi,\lambda))\cos^2\theta d\xi + \partial_\lambda \theta^{-\infty}(\lambda)\\
&<0.
\end{aligned}
\]
In other words, $\theta^u(0,\lambda)$ is a strictly decreasing function with respect to $\lambda\in(-\infty, 0]$. In addition, we claim that
\begin{equation}\label{e:bnd}
\theta^u(0,\lambda_0)>0, \quad \theta^u(0,0)=-\pi/2.
\end{equation}
As a result of the monotonicity and boundary conditions, there exists a unique $\lambda_*\in(\lambda_0, 0)$ such that $\theta^u(0,\lambda)=B_0=0$; that is, $\lambda_0$ is the only eigenvalue in the interval $(-\lambda_0,0)$.  We are now left to prove that \eqref{e:bnd} holds.
\begin{itemize}
\item $\theta^u(0,\lambda_0)>0$. As a matter of fact, for the interval $(-\infty, -\bar{\xi}(\lambda)]$, the analysis about the angle evolution of the unstable Jost solutions in Case 1 holds. More specifically, as $\xi$ goes from $-\infty$ to $\bar{\xi}(\lambda)$, $\theta^u(\xi,\lambda)$ decreases from $\theta^{-\infty}(\lambda)$ to $\bar{\theta}(\lambda):=\theta^u(\bar{\xi}(\lambda), \lambda)>0=\Theta_1^+(\bar{\xi}(\lambda),\lambda)$. For $\lambda=\lambda_0$, we have $\bar{\xi}(\lambda)$ and thus $\theta^u(0,\lambda_0)=\bar{\theta}(\lambda_0)>0$.
\begin{re}
The angle evolution of the unstable Jost solution, $\theta^u(\xi,\lambda)$, is a strictly decreasing function for $\xi\in(-\infty, 0]$.
For the interval $(-\infty, -\bar{\xi}(\lambda)]$, it is just shown. For the interval $(-\bar{\xi}(\lambda), 0]$, it is straightforward to see that $\theta_\xi(\xi, \lambda)=A(\xi,\lambda)\cos^2\theta-\sin^2\theta<0$,
which concludes our proof.
\end{re}
\item $\theta^u(0,0)=-\pi/2$. Note that $\lambda=0$ is proved to be an eigenvalue of \eqref{eigen} with eigenfunction
\[
q_e=(4-\partial_\xi^2)^{-1}\partial_\xi \phi (\xi),
\]
which satisfies the ODE
\begin{equation*}
\partial_\xi^2q_e-4q_e=-\partial_\xi \phi (\xi),\quad q_e\in H^2(\R).
\end{equation*}
To show that $\theta^u(0,0)=-\pi/2$, we only need to prove that $q_e(\xi)$ has exactly one $0$ on $(-\infty,\infty)$. Note that $q_e(\xi)$ is an odd function since $\phi (\xi)$ is even. Therefore it suffices to show
\begin{equation}\label{target}
q_e(\xi)<0,\quad \mbox{ for all }\xi>0.
\end{equation}
In fact, as the only solution which decays on both $\pm\infty$,
\begin{equation*}
q_e(\xi)=\frac{1}{4}\left (\int_\xi^{+\infty}e^{2(\xi-s)}\partial_\xi \phi (s)ds+\int_{-\infty}^\xi e^{-2(\xi-s)}\partial_\xi \phi (s)ds \right ).
\end{equation*}
Integrating by parts and change variable yields that, for any $\xi>0$,
\begin{eqnarray*}
q_e(\xi)
&=&\frac{1}{2} \left ( \int_\xi^{+\infty}e^{2(\xi-s)}\phi (s)ds-\int_{-\infty}^\xi e^{-2(\xi-s)} \phi (s)ds \right ) \\
&=&\frac{1}{2} \left ( \int_\xi^{+\infty}(e^{2\xi}-e^{-2\xi})e^{-2s}\phi (s)ds-\int_{-\xi}^\xi e^{-2(\xi-s)}\phi (s)ds \right )\\
&<&\frac{1}{2} \left ( \int_\xi^{+\infty}(e^{2\xi}-e^{-2\xi})e^{-2s}\phi (\xi)ds-\int_{-\xi}^\xi e^{-2(\xi-s)}\phi (\xi)ds \right ) \\
&=&\frac{1}{2}\phi (\xi) \left ( \int_\xi^{+\infty}(e^{2\xi}-e^{-2\xi})e^{-2s}ds-\int_{-\xi}^\xi e^{-2(\xi-s)}ds \right ) \\
&=&0,
\end{eqnarray*}
\end{itemize}
which concludes the desired  statement (3) and hence completes the proof of Theorem  \ref {spectrum}.
\end{proof}

\section{Spectral Stability of the DP Smooth Solitons}\label{s:4}

In this section, we give a proof of Theorem \ref{stability},  which is mainly based on  the frame work of Lin and Zeng \cite{LZ}.

Let $X_-\triangleq \spn\{\phi_*\}$ be the eigenspace of the operator $L_c$ with respect to the unique negative eigenvalue $\lambda_*$ and
\begin{equation*}
X_+\triangleq (X_-\oplus Ker\, L_c)^{\perp_{(L_c\cdot,\cdot)_{L^2}}}.
\end{equation*}
We then conclude that the Morse index of $L_c$ is 1, denoted as $n^-(L_c)\triangleq dim\, X_-=1$, and have the following decomposition
\[
L^2(\R)=X_-\oplus ker\,  L_c\oplus X_+,\quad
\]
where all subspaces are invariant under $L_c$, satisfying
\begin{itemize}
\item $(L_cv,v)<0$ for all $v\in X_-\setminus \{0\};$

\item there exists $\delta>0$ such that
\[
(L_cv,v)\geq\delta\Vert v\Vert_{L^2(\R)},\quad \mbox{ for any } v\in X_+.
\]
\end{itemize}
Given all the above conditions, according to the work by Lin and Zeng \cite{LZ}, we have the following index inequality,
\begin{equation}\label{index}
k_0^{\leq0}\leq n^-(L_c),
\end{equation}
where $k_0^{\leq0}$ is the number of nonpositive dimensions of $( L_c\cdot, \cdot)$ restricted to the generalized kernel of $JL_c$ modulo $\ker (L_c)$. Please note that the above inequality is a direct consequence of the general index formula Eq (1.2) in \cite{LZ}; see Section 2.4 in \cite{LZ} for details. Moreover, we have the following lemma.
\begin{lemma}[Corollary 2.2, \cite{LZ}]\label{counting}
If $k_0^{\leq0}= n^-(L)$, then the corresponding linearized flow is spectrally stable; that is, there exists no exponentially unstable solution.
\end{lemma}
%
%\begin{re}
%It is shown in  \cite{LZ} that the index formula \eqref{index} holds under much more general conditions than we assume. In particular, $Ker\, L$ does not need to be of finite dimensional. Moreover, if the strict positive definiteness on $X_+$ is not available, some other general conditions also guarantee the index formula. The closed subspaces $\textcolor{red}{X_{\pm}}$ are not required to be $L$-invariant.
%\end{re}

To obtain the spectral stability result in Theorem  \ref{stability}, it suffices to  prove the following index equality.
\begin{lemma}\label{convex}
It holds that $k_0^{\leq0}=n^-(L_c)=1$ for any $c>2k>0$.
\end{lemma}

\begin{proof} Recall that
\[
\frac{\delta H}{\delta u}(\phi)+c\frac{\delta S}{\delta u}(\phi)=0,
\]
which, taken derivative with respect to $c$ to both sides, yields,
\[
L_c\partial_c\phi=-\frac{\delta S}{\delta u}(\phi)=-(1-\partial_\xi^2)(4-\partial_\xi^2)^{-1}\phi.
\]
Therefore, we denote the general kernel of $JL_c$ as $\mathrm{gKer}(JL_c)$, recall that $J=\partial_x(4-\partial_\xi^2)(1-\partial_\xi^2)^{-1}$ and have
\[
JL_c\partial_c\phi=-\partial_x\phi\in\ker(L_c)\subseteq\ker(JL_c),
\]
implying that $\partial_c\phi\in \mathrm{gKer}(JL_c)\backslash \ker(L_c)$. In addition, we have
\[
(L_c\partial_c\phi,\partial_c\phi)=(-\frac{\delta S}{\delta u}(\phi),\partial_c\phi)
=-\frac{d}{dc}S(\phi).
\]
As a result, it suffices to show
\begin{equation}\label{e:Sd}
\frac{d}{dc}S(\phi)>0
\end{equation}
to conclude $1\leq k_0^{\leq0}\leq n^-(L_c)=1$ and thus $k_0^{\leq0}= n^-(L_c)=1$.

To prove \eqref{e:Sd}, we first derive a more explicit expression of $S(\phi)$.
Denoting
\[
w =(4-\partial_\xi^2)^{-1}\phi,
\]
we have
\begin{equation}\label{P-mix}
S(\phi)=\frac{1}{2}\int_{-\infty}^\infty \phi \cdot(1-\partial_\xi^2)(4-\partial_\xi^2)^{-1}\phi  d\xi=\frac{1}{2}\int_{-\infty}^\infty \phi \cdot(\phi-3w)d\xi.
\end{equation}
The profile $w$ can be expressed in terms of $\phi$. More specifically,
the traveling wave equation  \eqref{traveling}:
\[
c(1-\partial_\xi^2)(4-\partial_\xi^2)^{-1}\phi-[\frac{1}{2}(\phi)^2+(4-\partial_\xi^2)^{-1}2k\phi]=0
\]
yields
\[
c(\phi-3w)=\frac{1}{2}\phi^2 +2kw,
\]
%or equivalently,
%\[
%2cw-2cw_{\xi\xi}=(4w-w_{\xi\xi})^2+4kw.
%\]
which, after simple rearrangements, yields
\[
w=\frac{2c\phi -\phi^2 }{6c+4k}\quad \mbox{ and }\quad
\phi-3w=\frac{3\phi+4k}{2(3c+2k)}\phi.
\]
Taking advantage of \eqref{P-mix} and the evenness of $\phi$, we then have
\begin{equation}\label{P-phi}
S(\phi)%=\frac{1}{2}\int_{-\infty}^\infty \phi \cdot(\phi-3\frac{2c\phi -\phi^2 }{6c+4k})d\xi%=\frac{1}{4(3c+2k)}\int_{-\infty}^\infty(3\phi +4k)\phi^2 d\xi
=\frac{1}{2(3c+2k)}\int_{-\infty}^0(3\phi +4k)\phi^2 d\xi.
\end{equation}

In order to derive a more explicit expression of the above integral with respect to $c$, we take advantage of \eqref{firstintegral}, which reads
\[
\frac{1}{2}(c-\phi )^2\phi _\xi^2=\phi^2 (\frac{1}{2}\phi^2 -c\phi +\frac{2}{3}k\phi +\frac{1}{2}c^2-kc)=\phi^2 P(\phi)=\frac{1}{2}\phi^2(\phi-\phi_+)(\phi-\phi_-),
\]
where we recall that $\phi_\pm=c-\frac{2}{3}k\pm\sqrt{\frac{2}{9}k(3c+2k)}$ are the two positive roots of the quadratic polynomial $P(\phi)$.
As a result, for $\xi\in(-\infty,0)$, $\phi _\xi>0$, and
\[
\phi =\frac{c-\phi }{\sqrt{(\phi-\phi_+)(\phi-\phi_-)}}\phi _\xi, \quad \mbox{ for }\xi\in(-\infty,0).
\]
Plugging this expression of $\phi$ on $\xi\in(-\infty,0]$ into \eqref{P-phi}, we have
\begin{equation}\label{P-phi-s}
\begin{aligned}
S(\phi )
&=\frac{1}{2(3c+2k)}\int_{-\infty}^0(3\phi +4k)\phi^2 d\xi\\
&=\frac{1}{2(3c+2k)}\int_{-\infty}^0(3\phi +4k)\phi \cdot \frac{c-\phi }{\sqrt{(\phi-\phi_+)(\phi-\phi_-)}}\phi _\xi d\xi\\
%&=&\frac{1}{2(3c+2k)}\int_0^\al\frac{(3\phi +4k)\phi (c-\phi )}{\sqrt{\phi^2 -2(c-\frac{2}{3}k)\phi +c^2-2ck}} d\phi \\
&=\frac{1}{2(3c+2k)}\int_0^{\phi_-}\frac{(3\phi +4k)\phi (c-\phi )}{\sqrt{(\phi-\phi_+)(\phi-\phi_-)}} d\phi.
\end{aligned}
\end{equation}
It is noted that one can not take derivative with respect to $c$ directly because of the singularity in the denominator. Instead, introducing the change of variable
\[
z\triangleq\sqrt{(\phi-\phi_+)(\phi-\phi_-)}, \quad \al_\pm\triangleq\frac{\phi_+\pm\phi_-}{2}, \quad \beta\triangleq\phi_+\phi_-,
\]
and noting that
\[
d\phi =-\frac{zdz}{\al_+-\phi},\quad  \al_+-\phi=\sqrt{z^2+\al_-^2}, \quad \al_+=c-\frac{2}{3}k, \quad \al_-=\sqrt{\frac{2}{3}kc+\frac{4}{9}k^2}, \quad \be=c^2-2kc,
\]
the expression of $S$ in \eqref{P-phi-s} becomes
\[
\begin{aligned}
S(\phi )
=&\frac{1}{2(3c+2k)}\int_0^{\phi_-}\frac{(3\phi +4k)\phi (c-\phi )}{\sqrt{(\phi-\phi_+)(\phi-\phi_-)}} d\phi\\
=&\frac{3}{2(3c+2k)}\int_0^{\sqrt{\be}}\frac{\big[(\al_+-\phi) -(\al_++\frac{4}{3}k)\big]\big[(\al_+-\phi)-\al_+\big]\big[(\al_+-\phi )-(\al_+-c)\big]}{\al_+-\phi } dz\\
=&\frac{3}{2(3c+2k)}\int_0^{\sqrt{\be}}\left\{(\al_+-\phi)^2-(3\al_++\frac{4}{3}k-c)(\al_+-\phi) \right. \\
&\left.+\left[3\al_+^2+2(\frac{4}{3}k-c)\al_+-\frac{4}{3}kc\right]-\frac{\al_+(\al_++\frac{4}{3}k)(\al_+-c)}{\al_+-\phi}\right\}dz\\
=&\frac{3}{2(3c+2k)}\int_0^{\sqrt{\be}}\left\{ (z^2+\al_-^2)-2(c-\frac{1}{3}k)\sqrt{z^2+\al_-^2}+\big(c^2-\frac{4}{3}kc-\frac{4}{9}k^2 \big)+\frac{\frac{2}{3}k(c^2-\frac{4}{9}k^2)}{\sqrt{z^2+\al_-^2}}\right\}dz\\
=&\frac{1}{2(3c+2k)}\left\{z^3+(3c^2-2kc)z+(k-3c)\left[z\sqrt{z^2+\al_-^2}+
\al_-^2\log(z+\sqrt{z^2+\al_-^2})\right] \right.\\
&\left.
+2k(c^2-\frac{4}{9}k^2)\log(z+\sqrt{z^2+\al_-^2})\right\}\bigg|_{z=0}^{z=\sqrt{\be}}\\
=&\frac{(c^2-ck-\frac{2}{3}k^2)\sqrt{c^2-2ck}}{2(3c+2k)}
-\frac{1}{9}k^2\log\frac{c-\frac{2}{3}k+\sqrt{c^2-2ck}}{\sqrt{\frac{2}{3}kc+\frac{4}{9}k^2}}.
\end{aligned}
\]
A straightforward but lengthy calculation shows that the derivative
\[
\frac{d}{dc}S(\phi )=\frac{3c^2(c+k)}{(3c+2k)^2}\sqrt{\frac{c-2k}{c}}>0,\quad \mbox{ for any }c>2k>0.
\]
This completes  the proof Lemma \ref {convex}. \end{proof}

\bigskip

\noindent{\bf Acknowledgments}  The authors thank Professor Chongchun Zeng for stimulating conversation. The work of Li is partially supported by the NSFC grant 11771161. The work of Liu is partially supported by the Simons Foundation  grant 499875. The work of Wu is partially supported by the NSF grant DMS-1815079.  The work is initiated during Li's  visiting to the University of Texas at Arlington. Li thanks UT Arlington for great hospitality. Most of the work were done when  Liu and Wu were visiting Department of Mathematics,  Huazhong University of Science and Technology, who would like to thank the department for its warm hospitality.

\bibliographystyle{plain}

\end{document}